\documentclass{article}
\usepackage{amsmath,amsfonts,amssymb,amsthm}
\usepackage{epsfig}
\usepackage{mathrsfs}
\usepackage{enumerate}
\newtheorem{theorem}{Theorem}
\newtheorem{lemma}[theorem]{Lemma}
\newtheorem{koro}[theorem]{Corollary}

\newcommand{\bd}{{\mathrm{bd}}\,}

\newcommand{\relint}{{\mathrm{relint}}\,}
\newcommand{\cl}{{\mathrm{cl}}\,}

\newcommand{\Nor}{{\mathrm{Nor}}\,}

\newcommand{\R}{{\mathbb R}}

\newcommand{\E}{{\mathrm e}}

\newcommand{\Q}{{\mathbb Q}}

\newcommand{\cK}{{\mathcal  K}}
\newcommand{\cP}{{\mathcal  P}}

\newcommand{\gomega}{\tilde\omega}

\catcode`@=11
\def\section{%
\setcounter{equation}{0} \setcounter{theorem}{0} \@startsection
{section}{1}{\z@}{-4.0ex plus -1ex minus
    -.2ex}{2.3ex plus .2ex}{\bf}}
\catcode`@=12 \theoremstyle{definition}
\newtheorem{example}{Example}

\begin{document}

\title{Integral geometry of translation invariant functionals, II: The case of general convex bodies}

\author{
Wolfgang Weil\\
 \\ {\it Department of Mathematics, 
 Karlsruhe Institute of Technology}\\
 {\it 76128 Karlsruhe, Germany}}
\date{\today}

\maketitle
\begin{abstract}
\noindent In continuation of Part I, we study translative integral formulas for certain translation invariant functionals, which are defined on general convex bodies. Again, we consider local extensions and  use these to show that the translative formulas extend to arbitrary continuous and translation invariant valuations. Then, we discuss applications to Poisson particle processes and Boolean models which contain, as a special case, some new results for flag measures.
\end{abstract}

\section{Introduction}

In the introduction to the first part of this paper \cite{W09}, we have motivated the study of translation invariant functionals $\varphi$ on ${\cal K}$, the class of convex bodies in $\R^d$, which have a local extension as a kernel. For functionals on ${\cal P}$, the set of polytopes, these local functionals provide a richer class than the classical additive functionals ({\it valuations}).  A major reason for studying local functionals was the fact that they obey translative integral formulas similar to the integral geometric results for curvature measures and intrinsic volumes. We have seen that there are some interesting new functionals contained in this local concept, for example the total $k$-volume of the $k$-skeleton of a polytope, $k\in\{ 0,...,d-1\}$. Although the local functionals on ${\cal P}$ need not be additive, they allow an extension to unions of polytopes which are in mutual general position. This was used in \cite{W09} to 
obtain formulas for local functionals of Boolean models which are in analogy to the well-known results for intrinsic volumes (as they are presented in \cite[Section 9.1]{SW}).

In this second part, we discuss local functionals $\varphi$ on ${\cal K}$. It is natural then to add a continuity condition (in order to use polytopal approximation, for example). As we shall see, this changes the situation drastically. Namely, a (continuous) local functional on ${\cal K}$ is automatically additive, hence a valuation. In the opposite direction, for any translation invariant continuous valuation $\varphi$ on ${\cal K}$, the restriction to ${\cal P}$ is a local functional. However, this does not guarantee that $\varphi$ admits a local extension on ${\cal K}$. In fact, the question whether every translation invariant continuous valuation on ${\cal K}$ is a local functional remains open.

We shall show, that there is again a translative integral formula for local functionals on ${\cal K}$ which can be iterated. This even holds for valuations (without the assumption on a local extension). Thus, for a translation invariant, continuous and $j$-homogeneous valuation $\varphi^{(j)}$ on ${\cal K}$, we obtain a sequence of mixed functionals $\varphi^{(j)}_{m_1,...,m_k}$ which are in analogy to the mixed functionals $V^{(j)}_{m_1,...,m_k}$ arising from the intrinsic volumes $V_j, j=0,...,d-1$. We then study kinematic formulas and we extend some of the results from the first part to mean values of valuations for Boolean models with general convex or polyconvex grains. Finally, we discuss a few special cases and obtain new integral and mean value formulas for flag measures. In an appendix we provide an approximation result (with a proof by Rolf Schneider) which is used in Section 3.

\section{Definitions and basic result}

For completeness, we repeat some of the notations and definitions used in \cite{W09}. For general notions from convex geometry, we refer to \cite{S}.

Let $\cal K$ be the space of  convex bodies in $\R^d$ supplied with the Hausdorff metric, let $\cal P$ be the (dense) subset of convex polytopes and let ${\cal B}$ denote the $\sigma$-algebra of Borel sets in $\R^d$. For a convex body $K$ and $j=0,...,d$, $V_j(K)$ denotes the $j$th intrinsic volume and $\Phi_j(K,\cdot)$ is the $j$th curvature measure. Thus, $V_d(K)=\lambda (K)$ is the volume of $K$ (and $\lambda$ is the Lebesgue measure in $\R^d$). We denote by $\lambda_K$ the restriction of $\lambda$ to $K$. For a polytope $P$, let ${\cal F}_j(P)$ be the collection of $j$-faces of $P$, $j=0,...,d-1$, and let $n(P,F)$, for a face $F\in{\cal F}_j(P)$, be the intersection of the normal cone $N(P,F)$ of $P$ at $F$ with the unit sphere $S^{d-1}$;  this is a member of ${\wp}_{d-j-1}^{d-1}$, the class of $(d-j-1)$-dimensional spherical polytopes.  Later, we will also use the larger class $\tilde{\wp}_{d-j-1}^{d-1}$, which consists of the spherical polytopes of dimension $\le d-j-1$. 
For $F\in {\cal F}_j(P)$, let $\lambda_F$ be the restriction to $F$ of the ($j$-dimensional) Lebesgue measure in the affine hull of $F$. Here, the dimension of $\lambda_F$ will always be clear from the context. 

We call a functional $\varphi : {\cal K}\to \R$ {\it local}, if it has a {\it local extension}  $\Phi :  {\cal K}\times {\cal B} \to \R$, which is a kernel (a measurable function on ${\cal K}$ in the first variable and a finite signed Borel measure on ${\R}^d$ in the second variable) and such that $\Phi$ has the following properties:
\begin{itemize}
\item $\varphi (K) = \Phi(K,{\mathbb R}^d )$ for all $K\in {\cal K}$,
\item $\Phi$ is {\it translation covariant}, that is, satisfies $\Phi (K+x,A+x) = \Phi (K,A)$ for $K\in{\cal K}$, $A\in{\cal B}$, $x\in {\mathbb R}^d$,
\item $\Phi$ is {\it locally determined}, that is, $\Phi (K,A)=\Phi (M,A)$ for $K,M\in{\cal K}$, $A\in{\cal B}$, if there is an open set $U\subset {\mathbb R}^d$ with $K\cap U=M\cap U$ and $A\subset U$,
\item  $K\mapsto \Phi(K,\cdot)$ is weakly continuous on ${\cal K}$ (w.r.t. the Hausdorff metric).
\end{itemize}
 
By definition, local functionals $\varphi$ are translation invariant and continuous. Obviously, for a local functional $\varphi$ on ${\cal K}$, the restriction to ${\cal P}$ is a local functional on ${\cal P}$ in the sense of \cite{W09}, and $\varphi$ and its local extension $\Phi$ are continuous on ${\cal P}$. In the other direction, it is not clear which local  functionals on ${\cal P}$ with suitable continuity properties can be extended to  local  functionals on ${\cal K}$. Since we will show that local functionals on ${\cal K}$ are valuations, additivity is a necessary precondition for a functional on ${\cal P}$ to have a continuous extension to ${\cal K}$. For valuations on ${\cal P}$ the extension problem under various continuity conditions is discussed in the recent paper \cite{HHW}. We recall here that $\varphi : \cal S \to \cal X$ (where $\cal S$ is either $\cal K$ or $\cal P$ or $\tilde{\wp}_{d-j-1}^{d-1}$ and where $\cal X$ is an (additively written) Abelian group) is {\it additive} (in case $\cal X =\R$ we also speak of a {\it valuation}), if
$$
\varphi (K\cup M) + \varphi (K\cap M) = \varphi (K) + \varphi (M)
$$
holds for all $K,M,K\cup M\in\cal S$. The function $\varphi$ is {\it simple}, if $\varphi (K)=0$ whenever $K\in\cal S$ has dimension $\le d-1$ (if $\cal S$ is  $\cal K$ or $\cal P$), respectively $\le d-j-2$ (if ${\cal S} = \tilde{\wp}_{d-j-1}^{d-1}$). We also mention that $\varphi : \cal S\to\cal X$ is called {\it weakly additive}, if it satisfies
$$
\varphi (K) + \varphi (K\cap H) = \varphi (K\cap H^+) +\varphi (K\cap H^-)
$$
for all $K\in\cal S$ and all hyperplanes $H$ (meeting $K$) with corresponding halfspaces $H^+,H^-$. (In case ${\cal S} = \tilde{\wp}_{d-j-1}^{d-1}$ the hyperplanes are great spheres and the halfspaces are halfspheres.)

Let $\varphi$ be a real-valued functional on ${\cal K}$ and $k\in\{ 0,...,d \}$. We call $\varphi$ {\it $k$-homogeneous}, if $\varphi (\alpha K) = \alpha^k \varphi (K)$ holds for all $\alpha\ge 0$ and all $K\in {\cal K}$. Similarly, for a measure-valued functional $\Phi$ (like a local extension of $\varphi$), $k$-homogeneity means that $\Phi (\alpha K, \alpha A) = \alpha^k \Phi (K,A)$ holds for all $\alpha \ge 0$, all $K\in {\cal K}$ and all Borel sets $A\subset \R^d$.

We now formulate a first and basic result of the paper. 

\begin{theorem}\label{th1} Let $\varphi$ be a local functional on $\cal K$ with local extension $\Phi$. Then $\varphi$ has a unique representation
\begin{equation}\label{func1}
\varphi (K) = \sum_{j=0}^{d-1} \varphi^{(j)}(K) + c_dV_d(K)
\end{equation}
with $j$-homogeneous local functionals $\varphi^{(j)}$ on $\cal K$ and a constant $c_d\in\R$. Moreover, there is a unique decomposition
\begin{equation}\label{func2}
\Phi (K,\cdot) = \sum_{j=0}^{d-1} \Phi^{(j)}(K,\cdot) + c_d\lambda_K 
\end{equation}
such that $\Phi^{(j)}$ is a local extension of $\varphi^{(j)}$, for $j=0,...,d-1$.

For a polytope $P$, each $\Phi^{(j)}$ has the form
\begin{equation}\label{func3}
\Phi^{(j)} (P,\cdot) = \sum_{F\in {\cal F}_j(P)} f_j(n(P,F))\lambda_F
\end{equation}
with a (uniquely determined) 
simple additive function $f_j$ on $\tilde{\wp}_{d-j-1}^{d-1}$. We put $f_d=c_d$ and call $f_0,...,f_d$ the {\em associated functions} of $\Phi$. 

As a consequence, $\Phi^{(j)} (K,\cdot)$ and $\Phi (K,\cdot)$ depend additively on $K\in{\cal K}$ and $\varphi^{(j)}$ and $\varphi$ are valuations.

If $\Phi\ge 0$, then $f_j\ge 0, j=0,...,d$, and thus $\Phi^{(j)}\ge 0, \varphi^{(j)}\ge 0$, for $j=0,...,d-1$, and $\varphi\ge 0$.
\end{theorem}

We remark, that the local extension $\Phi$ of a local functional $\varphi$ need not be unique. Even more, for each $j=1,...,d-1$ and each $j$-homogeneous kernel $\Phi^{(j)}\ge 0$, there is a $j$-homogeneous kernel $\tilde\Phi^{(j)}\ge 0$, with $\Phi^{(j)}\not = \tilde\Phi^{(j)}$ and such that $\Phi^{(j)}$ and $\tilde\Phi^{(j)}$  are both local extensions of the same local functional  $\varphi^{(j)} = \Phi^{(j)}(\cdot,\R^d) = \tilde\Phi^{(j)}(\cdot ,\R^d)$. The proof of this fact follows the argument given in \cite{W09} for the polytopal case and will be explained in the next section.

\section{Properties of local functionals}

In this section, we first give the proof of Theorem \ref{th1}.

\begin{proof}
Let $\varphi$ be a local functional on ${\cal K}$ and let $\Phi$ be a local extension of $\varphi$. The restriction of $\varphi$ to ${\cal P}$ is a local functional on ${\cal P}$ and the restriction of $\Phi$ to ${\cal P}$ is a corresponding local extension. Thus, Theorem 2.1 of \cite{W09} holds for these restrictions and this yields \eqref{func1}, \eqref{func2} and \eqref{func3} for polytopes $K$ resp. $P$, where $f_j$ is a measurable function on ${\wp}_{d-j-1}^{d-1}$. We extend it to $\tilde{\wp}_{d-j-1}^{d-1}$ by $f_j(p) =0$, for $p\in \tilde{\wp}_{d-j-1}^{d-1}\setminus {\wp}_{d-j-1}^{d-1}$.

The continuity property of $\varphi$ and $\Phi$ on ${\cal K}$ and the polynomial expansion 
\begin{equation*}
\Phi (\alpha P,\alpha A) = \sum_{j=0}^{d-1} \alpha^j \Phi^{(j)}(P,A) + \alpha^d c_d\lambda (A\cap P)
\end{equation*}
which was proved in \cite{W09} for polytopes $P$, $\alpha\ge 0$ and Borel sets $A\in{\cal B}$,
 show that $\varphi^{(j)}$ and $\Phi^{(j)}$ can be extended continuously to ${\cal K}$, for $j=0,...,d$. 
The equations \eqref{func1} and \eqref{func2} then arise from corresponding results in \cite{W09} by approximation with polytopes and in the same way the non-negativity properties follow.

It remains to show the simple additivity of $f_j$. This then implies that $\Phi^{(j)}$ is additive on ${\cal P}$ and therefore also $\Phi, \varphi^{(j)}$ and $\varphi$. Additivity on ${\cal K}$ then follows by approximation, due to the continuity of the functionals and using Lemma \ref{Rolf} in the Appendix. Let $p\in {\wp}_{d-j-1}^{d-1}$ be a spherical polytope. We can find a polytope $P\in{\cal P}$ and a face $F\in{\cal F}_j(P)$ such that $p=n(P,F)$. Let $A\subset \R^d$ be an open set with $A\cap F\subset \relint F$, $\lambda_F(A)>0,$ and $\cl A\cap G=\emptyset$ for all $G\in{\cal F}_j(P), G\not= F$. It follows that $\Phi^{(j)}(P,\bd A) = f_j(p)\lambda_F(\bd A)=0.$ Let $u\in F^\bot$ be a unit vector such that $u^\bot$ divides $p$ into the two pieces $p_1$ and $p_2$, $p_1,p_2\in {\wp}_{d-j-1}^{d-1}$. We consider the polytope $P_t = P + t[0,u], t>0$. For small enough $t$, this polytope has the two $j$-faces $F$ and $F_t = F+ tu$ and $n(P_t,F) = p_1, n(P_t,F_t)=p_2$, say. Moreover, as $t\to 0$, we have $F_t\to F$ in the Hausdorff metric and so $\lambda_{F_t}\to \lambda_F$ weakly. Since we may assume that, for small $t$, the closure $\cl A$ does not meet any $j$-faces of $P_t$ other than $F$ and $F_t$ and since $\bd A$ is a $0$-set for  $\Phi^{(j)} (P,\cdot)$, the weak convergence $\Phi^{(j)} (P_t,\cdot)\to \Phi^{(j)} (P,\cdot)$ (as $t\to 0$) and the Portmanteau theorem yield
$$\Phi^{(j)} (P_t,A)\to \Phi^{(j)} (P,A)$$
(as well as $\lambda_{F_t}(A)\to \lambda_{F}(A)$).
But 
$$\Phi^{(j)} (P_t,A) = f_j(p_1)\lambda_F(A) + f_j(p_2)\lambda_{F_t}(A)$$
and
$$\Phi^{(j)} (P,A) = f_j(p)\lambda_F(A).$$
Therefore,
$$f_j(p) = f_j(p_1) + f_j(p_2),$$
which shows that $f_j$ is weakly additive (and simple) on ${\wp}_{d-j-1}^{d-1}$. Due to $f_j(p) =0$, for $p\in \tilde{\wp}_{d-j-1}^{d-1}\setminus {\wp}_{d-j-1}^{d-1}$, these properties extend to $\tilde{\wp}_{d-j-1}^{d-1}$.   Since every weakly additive functional on $\tilde{\wp}_{d-j-1}^{d-1}$ is additive (see \cite[Lemma 1.3]{McM93} or \cite[Theorem 6.2.3]{S}), the proof is complete.
\end{proof}

Concerning the non-uniqueness of local extensions, the following example was discussed in \cite{W09} for local functionals on ${\cal P}$. We explain shortly the generalization to functionals on ${\cal K}$. For $K\in{\cal K}$, we use the support measures $\Lambda_j(K,\cdot), j=0,\dots, d-1,$ which are finite Borel measures on the (generalized) normal bundle $\Nor K$ of $K$, normalized such that $\Lambda_j(K,\Nor K)$ is the $j$th intrinsic volume $V_j(K)$ of $K$. Recall that $\Nor K$ consists of all pairs $(x,u), x\in \bd K, u$ an outer normal vector to $K$ at $x$. The image of $\Lambda_j(K,\cdot)$ under the mapping $(x,u)\mapsto x$ is the $j$th curvature measure $\Phi_j(K,\cdot)$ and the image under $(x,u)\mapsto u$ is the $j$th area measure $\Psi_j(K,\cdot )$ of $K$. Recall also that $\Psi_j(K,\cdot )$ has centroid $0$. 

\begin{example} For $j\in\{ 1,...,d-1\}$, we consider the valuation $V_j : {\cal K}\to [0,\infty )$, $K\mapsto V_j(K)$, which has a local extension, namely $\Phi_j : K\mapsto \Phi_j (K,\cdot)$. Let $l = \langle \cdot,x_0\rangle$ be a linear function, $x_0\in\R^d, x_0\not= 0$.  Then $\tilde \Phi_j$, given by
$$
\tilde \Phi_j(K,A) = \Phi_j(K,A) + \int_{\Nor K} {\bf 1}_A(x) l(u) \Lambda_j(K, d(x,u))
$$
for $K\in{\cal K}$ and $A\in{\cal B}$, is another local extension, different from $\Phi_j$, since
\begin{align*}
\tilde\varphi_j (K) &= \tilde\Phi_j (K,\R^d) = \Phi_j (K,\R^d)+ \int_{S^{d-1}} l(u) \Psi_j(K, du)\cr
&= \Phi_j (K,\R^d) = V_j (K) .
\end{align*}
Moreover, if $\| x_0\|\le 1$, then
\begin{align*}
\tilde \Phi_j(K,A) &= \Phi_j(K,A) + \int_{\Nor K} {\bf 1}_A(x) l(u) \Lambda_j(K, d(x,u))\cr
&= \int_{\Nor K} {\bf 1}_A(x) (1+l(u)) \Lambda_j(K, d(x,u))\ge 0 
\end{align*}
such that both local extensions $\Phi_j$ and $\tilde\Phi_j$ are nonnegative (and additive). 
\end{example}

In \cite[Section 11.1]{SW}, a continuous, translation invariant valuation $\varphi$ on ${\cal K}$ was called a {\it standard functional} and standard functionals admitting a local extension were considered (in addition to the conditions which we imposed on a local extension, non-negativity was also required in \cite{SW}). As a consequence of Theorem \ref{th1}, every local functional $\varphi$ on ${\cal K}$ is a standard functional (with local extension). A major open problem concerns the opposite question. Is every standard functional $\varphi$ on ${\cal K}$ a local functional? Since the restriction of $\varphi$ to ${\cal P}$ has a local extension, this question is closely connected to the problem which standard functionals on $\cal P$ can be extended to standard functionals on $\cal K$. Variants of this, apparently open, problem are discussed in \cite{HHW}, where continuity on ${\cal P}$ with respect to the Hausdorff metric is replaced by appropriate continuity conditions on the associated functions $f_j$ appearing in \eqref{func3}. In particular, there are positive and negative results in \cite{HHW} using different flag measures of convex bodies. The following result is a simple outcome of these considerations. We present it here for completeness, together with its proof.

Let $G(d,j)$ denote the Grassmannian of $j$-dimensional subspaces of $\R^d$, $j\in\{0,...,d-1\}$. We say that a function $f_j$ on ${\wp}_{d-j-1}^{d-1}$ has a {\it continuous density} $h_j$, if $h_j : S^{d-1} \to [0,\infty)$ is a continuous function such that
\begin{equation}\label{densities}
f_j (p) = \int_{L^\bot\cap S^{d-1}} {\bf 1}_p(u)h_j(u) \gomega_{L^\bot}(du)
\end{equation}
for all $L\in G(d,j)$ and all $p\in {\wp}_{d-j-1}^{d-1}$ with $p\subset L^\bot$. Here, $\gomega_{L^\bot}$ denotes the normalized spherical Lebesgue measure in $L^\bot$. 

\begin{theorem}\label{extension}
Let $\varphi$ be a standard functional on ${\cal P}$ with associated functions $f_0,...,f_d$ which have continuous densities $h_0,...,h_d$ ($h_d=f_d=c_d$). Then $\varphi$ can be extended to a standard functional $\varphi$ on ${\cal K}$ and the latter has a local extension $\Phi$ given by 
$$
\Phi (K,A) = \sum_{j=0}^{d-1} \int_{\Nor K} {\bf 1}_A(x)h_j(u) \Lambda_j(K,d(x,u)) + c_d\lambda_K (A)
$$
for $K\in{\cal K}$  and  $A\in{\cal B}$.
\end{theorem}
 
\begin{proof} Let $P$ be a polytope. For a face $F$ of $P$, let $F^\bot$ be the linear space orthogonal to $F$. Under our assumptions, the local extension $\Phi$ of $\varphi$ on ${\cal P}$ satisfies
\begin{align*}
\Phi &(P,A) =  \sum_{j=0}^{d-1} \sum_{F\in{\cal F}_j(P)} f_j(n(P,F)) \lambda_F(A)+ c_d\lambda_P(A) \cr
&=  \sum_{j=0}^{d-1} \sum_{F\in{\cal F}_j(P)} \int_{F^\bot\cap S^{d-1}} {\bf 1}_{n(P,F)}(u)h_j(u) \gomega_{F^\bot}(du) \lambda_F(A)+ c_d\lambda_P(A)\cr
&=  \sum_{j=0}^{d-1} \sum_{F\in{\cal F}_j(P)} \int_A\int_{F^\bot\cap S^{d-1}} {\bf 1}_{n(P,F)}(u)h_j(u) \gomega_{F^\bot}(du)\lambda_F(dx)+ c_d\lambda_P(A)\cr
&=  \sum_{j=0}^{d-1} \int_{\Nor P}{\bf 1}_A(x)h_j(u) \Lambda_j(P,d(x,u)) + c_d\lambda_P(A)
\end{align*}
due to the representation of support measures of polytopes (see \cite[(4.3) and (4.18)]{S}).

Now let $K\in{\cal K}$ and $P_k\to K$ be a sequence of polytopes converging to $K$. We define a measure $\Phi(K,\cdot)$ on $\R^d$ by
$$
\Phi(K,A) = \sum_{j=0}^{d-1} \int_{\R^d\times S^{d-1}}{\bf 1}_A(x)h_j(u) \Lambda_j(K,d(x,u)) + c_d\lambda_K(A), \quad A\in{\cal B},
$$
and put $\varphi (K) = \Phi(K,\R^d)$. Then it follows from the weak continuity of support measures that $\Phi(P_k,\cdot)\to\Phi(K,\cdot)$ weakly, and thus also $\varphi(P_k) \to \varphi (K).$ It is easy to see that, in this way, $\varphi$ extends to a standard functional on ${\cal K}$ and $\Phi$ extends to its local extension. \end{proof}

Theorem \ref{extension} shows, in particular, that a standard functional $\varphi$ on ${\cal K}$ admits a local extension, if the restriction of $\varphi$ to ${\cal P}$ has associated functionals $f_0,...,f_d$ with continuous densities. The latter is the case if and only if the homogeneous parts $\varphi^{(0)},...,\varphi^{(d-1)}$ of $\varphi$ have continuous densities $h_0,...,h_{d-1}$, in the sense that
\begin{equation}\label{flagcont}
\varphi^{(j)}(K) = \int_{S^{d-1}}h_j(u) \Psi_j(K,du)
\end{equation}
for all $K\in{\cal K}$.

In order to get a more general result, it would be natural to allow continuous densities $h_j$ of $f_j$ in \eqref{densities} which may depend on the subspace $L$,
$$
f_j (p) = \int_{L^\bot\cap S^{d-1}} {\bf 1}_p(u)h_j(u,L) \omega_{L^\bot}(du).
$$
However, as was shown by an example in \cite{HHW}, this so-called flag continuity of a standard functional $\varphi$ on $\cal P$ is in general not sufficient for an extension to $\cal K$. Therefore, as a variant, a strong flag continuity was defined in \cite{HHW}, which guaranteed the existence of an extension. Strong flag continuity requires that the function $h_j$ on 
$$F(d,d-j)=\{ (u,L) \in S^{d-1}\times G(d,d-j) : u\in L\}$$
lies in the image of a certain integral  transform on $F(d,d-j)$. As a generalization of \eqref{flagcont}, this implies a representation
$$
\varphi^{(j)}(K) = \int_{F(d,d-j)}h_j(u,L) \psi_j(K,d(u,L))
$$
for all $K\in{\cal K}$, where $\psi_j(K,\cdot)$ is the $j$th flag measure of $K$. For details, we refer to \cite{HHW}. We will discuss flag measures again in Section 7, but mention here already that a strongly flag continuous standard functional $\varphi$ on $\cal K$ has a local extension.

\section{Translative integral formulas}

For local functionals $\varphi$ on ${\cal K}$ with local extension $\Phi$ and $j$-homogeneous parts $\Phi^{(j)}$, $j=0,...,d$, a translative integral formula follows by approximation with polytopes, parallel to the treatment of curvature measures and intrinsic volumes in Sections 5.2 and 6.4 of \cite{SW} and based on the corresponding result for polytopes (Theorem 5.1 in \cite{W09}). In the following result, we therefore leave out some parts of the proof and concentrate on the approximation argument. As in \cite{W09}, we denote the translate $A+x$ of a set $A\subset\R^d$ by $A^x$. We also define $\Phi( \emptyset,\cdot)=0$, and thus $\varphi(\emptyset)=0$.

\begin{theorem}\label{trans3} Let $\varphi$ be a local functional on ${\cal K}$, let $\Phi$ be a local extension and let $\varphi^{(j)}, \Phi^{(j)}$ be the $j$-homogeneous parts of $\varphi$ and $\Phi$, $j=0,\dots ,d$. Then, 
for $k\ge 2$, convex bodies $K_1,...,K_k \in {\cal K}$ and Borel sets  $A_1,...,A_k \in {\cal B}$, there are mixed measures $\Phi^{(j)}_{m_1,\dots,m_k}(K_1,\dots,K_k;\cdot)$ on ${({\mathbb R}^d)^{k}}$ such that
\begin{align}
&  \int_{({\mathbb R}^d)^{k-1}}
\Phi^{(j)}(K_1\cap K_2^{x_2}\cap \dots \cap K_k^{x_k},A_1\cap A_2^{x_2}\cap \dots \cap A_k^{x_k})\, \lambda^{k-1} (d(x_2,\dots ,x_k))\nonumber \\
& = \sum_{{m_1,\dots,m_k=j}\atop{m_1+\dots +m_k=(k-1)d+j}}^d
\Phi^{(j)}_{m_1,\dots,m_k}(K_1,\dots,K_k;A_1\times \dots \times A_k).\label{transint2}
\end{align}

The measure $\Phi^{(j)}_{m_1,\dots,m_k}(K_1,\dots,K_k;\cdot)$ depends continuously on
$K_1,...,K_k \in {\cal K}$ and is homogeneous of degree $m_i$  in $K_i$. For polytopes $K_1,...,K_k$, it coincides with the one appearing in Theorem 5.1 of \cite{W09}.
 
The total measures $\varphi^{(j)}_{m_1,\dots,m_k}(K_1,\dots,K_k) = \Phi^{(j)}_{m_1,\dots,m_k}(K_1,\dots,K_k;(\R^d)^k)$ satisfy the iterated translative formula
\begin{align}
&  \int_{({\mathbb R}^d)^{k-1}}
\varphi^{(j)}(K_1\cap K_2^{x_2}\cap \dots \cap K_k^{x_k})\, \lambda^{k-1} (d(x_2,\dots ,x_k))\nonumber \\
& = \sum_{{m_1,\dots,m_k=j}\atop{m_1+\dots +m_k=(k-1)d+j}}^d
\varphi^{(j)}_{m_1,\dots,m_k}(K_1,\dots,K_k).\label{transint3}
\end{align}
\end{theorem}

\begin{proof} If $K_1,...,K_k$ are polytopes, the assertions follow from Theorem 5.1 and Corollary 5.2 in \cite{W09}. Therefore, in this case, the mixed measures\linebreak $\Phi^{(j)}_{m_1,\dots,m_k}(K_1,\dots,K_k;\cdot)$ exist and the iterated translative formula for the local extension $\Phi^{(j)}$ holds for all Borel sets $A_1,...,A_k$ and is equivalent to
\begin{align}\label{6.4.4}
&  \int_{({\mathbb R}^d)^{k-1}} \int_{{\mathbb R}^d}f(x_1,x_1-x_2,\dots, x_1-x_k)\,\Phi^{(j)}(K_1\cap K_2^{x_2}\cap \dots \cap K_k^{x_k},d x_1)\nonumber\\
& \quad\times\, \lambda^{k-1} (d (x_2,\dots ,x_k)) \\
&  = \sum_{{m_1,\dots,m_k=j}\atop{m_1+\dots +m_k=(k-1)d+j}}^d \hspace{-3pt} \int_{({\mathbb R}^d)^k} f(x_1,\dots ,x_k)\,
\Phi^{(j)}_{m_1,\dots,m_k}(K_1,\dots,K_k;d(x_1,\dots ,x_k))\nonumber
\end{align}
for all continuous functions $f$ on $({\mathbb R}^d)^k$ (see \cite[formula (6.16)]{SW}, for the necessary arguments). 

As in \cite[p. 232]{SW}, we consider the functional $J(f,K_1,\dots,K_k)$ defined by the left side of \eqref{6.4.4}, for all $K_1,\dots,K_k\in {\cal K}$. Due to the weak continuity of the kernel $\Phi^{(j)}$, this functional depends continuously on $K_1,\dots,K_k$. Since the intersection of convex bodies is not a continuous operation, some additional arguments are necessary here, see \cite[p. 188]{SW}. One tool used there is the dominated convergence theorem which requires, for $K, M\in {\cal K}$, a $\lambda$-integrable upper bound of
$$
y\mapsto | \varphi^{(j)}(K\cap M^y)| .$$
We can use $C_K\cdot{\bf 1}_{K-M}(y)$ as such a bound, where 
$$C_K = \max\{ |\varphi^{(j)}(K')| : K'\subset K, K'\in{\cal K}\} ,$$ 
which is finite since $\varphi^{(j)}$ is continuous.

For $r_1,\dots,r_k > 0$ and a continuous function $f$ on $(\R^d)^k$, we define a continuous function $D_{r_1,\dots ,r_k}f$ on $({\mathbb R}^d)^k$  by
$$
D_{r_1,\dots ,r_k}f (x_1,\dots ,x_k) = f\left(\frac{x_1}{r_1},\dots, \frac{x_k}{r_k}\right)\qquad {\rm for}\ x_1,\dots, x_k\in {\mathbb R}^d .
$$
For polytopes $K_1,\dots ,K_k$, relation (\ref{6.4.4}) and the homogeneity properties of the mixed measures imply
\begin{align*}
&  J(D_{r_1,\dots, r_k}f,r_1K_1,\dots ,r_kK_k)\\
&  = \int_{({\mathbb R}^d)^{k-1}} \int_{{\mathbb R}^d}f\left(\frac{x_1}{r_1},\frac{x_1-x_2}{r_2},\dots, \frac{x_1-x_k}{r_k}\right)\\
&  \hspace*{4mm}\times\,\Phi^{(j)}(r_1K_1 \cap (r_2K_2)^{x_2}\cap\dots\cap (r_kK_k)^{x_k}, d x_1) \,\lambda^{k-1}(d(x_2,\dots,x_k))\\
&  = \sum_{{m_1,\dots,m_k=j}\atop{m_1+\dots +m_k=(k-1)d+j}}^d \\
&   \hspace*{4mm}\int_{({\mathbb R}^d)^k}
f\left(\frac{x_1}{r_1},\dots ,\frac{x_k}{r_k} \right)\, \Phi^{(j)}_{m_1,\dots,m_k}(r_1K_1,\dots,r_2K_k;d (x_1,\dots ,x_k)) \\
& = \sum_{{m_1,\dots,m_k=j}\atop{m_1+\dots +m_k=(k-1)d+j}}^d r_1^{m_1}\cdots r_k^{m_k}\\
&  \hspace*{4mm}\times\,\int_{({\mathbb R}^d)^k} f(x_1,\dots ,x_k)\,
\Phi^{(j)}_{m_1,\dots,m_k}(K_1,\dots,K_k;d (x_1,\dots ,x_k)).
\end{align*}

For arbitrary convex bodies $K_1,\dots ,K_k$, we choose sequences of polytopes $K_{1i},\dots K_{ki}$, $i \in{\mathbb N}$, such that $K_{1i}\to K_1$, \dots, $K_{ki}\to K_k$ for $i\to\infty$. Then
\[ J(D_{r_1,\dots ,r_k}f,r_1K_{1i},\dots ,r_kK_{ki}) \to J(D_{r_1,\dots ,r_k}f,r_1K_1,\dots,r_kK_k) \]
for every continuous function $f$ on $({\mathbb R}^d)^k$ and all $r_1,\dots ,r_k > 0$. From the polynomial expansion just established, we deduce the convergence of the coefficients
\[ \int_{({\mathbb R}^d)^k} f(x_1,\dots ,x_k)\,
\Phi^{(j)}_{m_1,\dots,m_k}(K_{1i},\dots,K_{ki};d (x_1,\dots ,x_k)) \]
and thus the weak convergence of the measures
\[ \Phi^{(j)}_{m_1,\dots,m_k}(K_{1i},\dots,K_{ki};\cdot ) \]
for $i \to \infty$. The limits, denoted by
$\Phi^{(j)}_{m_1,\dots,m_k}(K_{1},\dots,K_{k};\cdot ) $, are again finite
measures, satisfying
\begin{eqnarray}\label{6.4.5}
& & J(D_{r_1,\dots ,r_k}f,r_1K_1,\dots,r_kK_k) \nonumber\\
& &  = \sum_{{m_1,\dots,m_k=j}\atop{m_1+\dots +m_k=(k-1)d+j}}^d r_1^{m_1}\cdots r_k^{m_k}\nonumber\\
& & \hspace*{4mm}\times\,\int_{({\mathbb R}^d)^k} f(x_1,\dots ,x_k)\,
\Phi^{(j)}_{m_1,\dots,m_k}(K_1,\dots,K_k;d (x_1,\dots ,x_k)), \hspace*{7mm}
\end{eqnarray}
from which we see that they are independent of the approximating
sequences $(K_{1i})_{i\in{\mathbb N}},\dots, (K_{ki})_{i\in{\mathbb N}}$. For $r_1 = \dots = r_k= 1$, we obtain
(\ref{6.4.4}).

Thus, mixed measures for arbitrary bodies $K_1,\dots,K_k$ are defined which fulfill the iterated translation formula \eqref{transint2}. The formula \eqref{transint3} for the scalar functionals is a consequence.
\end{proof}

For further properties of the mixed measures $\Phi^{(j)}_{m_1,\dots,m_k}(K_1,\dots,K_k;\cdot)$ and mixed functionals $\varphi^{(j)}_{m_1,\dots,m_k}$, like symmetry, decomposability, translation covariance (resp. invariance) and local determination, we refer to \cite[Section 6.4]{SW}, where corresponding results are discussed for the mixed measures and functionals of intrinsic volumes and curvature measures. We just mention that the decomposability means that the mixed expressions split if one of the parameters $m_i$ is $d$. For example, $\Phi^{(j)}_{j,d}(K,M;\cdot)= \Phi^{(j)}(K,\cdot)\otimes \lambda_M$ and $\varphi^{(j)}_{j,d}(K,M)= \varphi^{(j)}(K)\lambda(M)$. In particular, we have
$$
\Phi^{(d)}_{d,d}(K,M;\cdot) = \Phi^{(d)}(K,\cdot)\otimes \lambda_M = c_d( \lambda_K\otimes\lambda_M)
$$
and
$$
\varphi^{(d)}_{d,d}(K,M) = c_dV_d(K)V_d(M).
$$
In view of the symmetry, the translative formulas for $k=2$ and $j<d$ therefore read
\begin{align}\label{2-case}
  \int_{{\mathbb R}^d} &\Phi^{(j)}(K \cap M^x, A\cap B^x)\,\lambda(dx)= \Phi^{(j)}(K,A) \lambda (M\cap B\nonumber)\\
&    + \sum_{m=j+1}^{d-1}\Phi_{m,d+j-m}^{(j)}(K,M;A \times B) +\lambda (K\cap A)\Phi^{(j)}(M,B)
\end{align}
and
\begin{align*}
 \int_{{\mathbb R}^d} &\varphi^{(j)}(K \cap M^x) \, \lambda(d x)
= \varphi^{(j)}(K)\lambda (M)\\
& + \sum_{m=j+1}^{d-1}\varphi_{m,d+j-m}^{(j)}(K,M) +
\lambda (K)\varphi^{(j)}(M) .
\end{align*}
We also emphasize that $\Phi^{(j)}_{m_1,\dots,m_k}(K_1,\dots,K_k;\cdot)$ and  $\varphi^{(j)}_{m_1,\dots,m_k}(K_1,\dots, K_k)$ are additive in each of the variables $K_1,\dots ,K_k$.

We remark that Theorem \ref{trans3} yields an extension of the results in \cite[Section 11.1]{SW}. There, for a standard functional $\varphi$ with local extension $\Phi$, associated kernels $\Phi_{(k)}$ were introduced through the $k$-fold translative integral. Theorem \ref{trans3} now shows that the associated kernels $\Phi_{(k)}$ can be developed into a sum of mixed kernels, as in the case of intrinsic volumes and curvature measures. We shall exploit this fact further in Section 6 when we investigate Boolean  models. 

By a slight variation of the argument in the above proof, we obtain a further extension, namely a translative integral formula and its iteration for standard functionals $\varphi$ on ${\cal K}$ which are not necessarily local.

\begin{theorem}\label{trans4} Let $\varphi$ be a standard functional on ${\cal K}$ and let $\varphi^{(j)}$ be its $j$-homogeneous part, $j=0,...,d$, with $\varphi^{(d)} =V_d$.  Then, 
for $k\ge 2$ and convex bodies $K_1,...,K_k \in {\cal K}$, there are mixed functionals $\varphi^{(j)}_{m_1,\dots,m_k}$ on ${({\cal K})^{k}}$ such that
\begin{align}
&  \int_{({\mathbb R}^d)^{k-1}}
\varphi^{(j)}(K_1\cap K_2^{x_2}\cap \dots \cap K_k^{x_k})\, \lambda^{k-1} (d(x_2,\dots ,x_k))\nonumber \\
& = \sum_{{m_1,\dots,m_k=j}\atop{m_1+\dots +m_k=(k-1)d+j}}^d
\varphi^{(j)}_{m_1,\dots,m_k}(K_1,\dots,K_k).\label{6.4.0}
\end{align}
The mapping $(K_1,...,K_k)\mapsto \varphi^{(j)}_{m_1,\dots,m_k}(K_1,\dots,K_k)$ is symmetric (w.r.t. permutations of the indices $1,...,k$), it is homogeneous of degree $m_i$ in $K_i$ and it is a standard functional in each of its variables $K_i$. 
\end{theorem}

\begin{proof} Due to Theorem 4.1 in \cite{W09}, the restriction of $\varphi^{(j)}$ to ${\cal P}$ is a local functional. Therefore, \cite[Corollary 5.2]{W09} implies that, for polytopes $K_1,...,K_k$, the mixed functionals $\varphi^{(j)}_{m_1,\dots,m_k}(K_1,...,K_k)$ exist and that the iterated translation formula \eqref{6.4.0} is satisfied. As in the proof of Theorem \ref{trans3} we consider the functional 
$$
J(K_1,...,K_k) = \int_{({\mathbb R}^d)^{k-1}}
\varphi^{(j)}(K_1\cap K_2^{x_2}\cap \dots \cap K_k^{x_k})\, \lambda^{k-1} (d(x_2,\dots ,x_k))
$$
for $K_1,...,K_k\in{\cal K}$ and show that it depends continuously on the bodies $K_i$ by using the dominated convergence theorem with the same upper bound
$$
 | \varphi^{(j)}(K\cap M^y)| \le C_K\cdot{\bf 1}_{K-M}(y)
$$
with 
$$C_K = \max\{ |\varphi^{(j)}(K')| : K'\subset K, K'\in{\cal K}\} .$$
The functional $J(K_1,...,K_k)$ corresponds to $J(1,K_1,...,K_k)$ in the proof of Theorem \ref{trans3}. Therefore, we obtain in the same manner, for polytopes $K_1,\dots ,K_k$, 
\begin{align*}
&  J(r_1K_1,\dots ,r_kK_k)\\
& = \int_{({\mathbb R}^d)^{k-1}} \varphi^{(j)}(r_1K_1 \cap (r_2K_2)^{x_2}\cap\dots\cap (r_kK_k)^{x_k}) \,\lambda^{k-1}(d(x_2,\dots,x_k))\\
& = \sum_{{m_1,\dots,m_k=j}\atop{m_1+\dots +m_k=(k-1)d+j}}^d 
 \varphi^{(j)}_{m_1,\dots,m_k}(r_1K_1,\dots,r_2K_k) \\
& = \sum_{{m_1,\dots,m_k=j}\atop{m_1+\dots +m_k=(k-1)d+j}}^d r_1^{m_1}\cdots r_k^{m_k}
\varphi^{(j)}_{m_1,\dots,m_k}(K_1,\dots,K_k).
\end{align*}

For arbitrary convex bodies $K_1,\dots ,K_k$, we choose again sequences of polytopes $K_{1i},\dots K_{ki}$, $i \in{\mathbb N}$, with $K_{1i}\to K_1$, \dots, $K_{ki}\to K_k$ for $i\to\infty$. Then, the continuity of $J$ and  the polynomial expansion just established imply the convergence of the mixed functionals
\[ \varphi^{(j)}_{m_1,\dots,m_k}(K_{1i},\dots,K_{ki} ) \]
for $i \to \infty$. The limit functionals
$\varphi^{(j)}_{m_1,\dots,m_k}(K_{1},\dots,K_{k}) $ satisfy
\begin{align*}
 &J(r_1K_1,\dots,r_kK_k)  = \sum_{{m_1,\dots,m_k=j}\atop{m_1+\dots +m_k=(k-1)d+j}}^d r_1^{m_1}\cdots r_k^{m_k} \varphi^{(j)}_{m_1,\dots,m_k}(K_1,\dots,K_k), 
\end{align*}
and are thus independent of the approximating
sequences $(K_{1i})_{i\in{\mathbb N}},\dots, (K_{ki})_{i\in{\mathbb N}}$. For $r_1 = \dots = r_k= 1$, we obtain
(\ref{6.4.0}).
\end{proof}

\section{Kinematic formulas}

Hadwiger's general integral theorem (see \cite[Theorem 5.1.2]{SW}) shows that, for an additive and continuous functional $\varphi$ on ${\cal K}$ and convex bodies $K,M$,
\begin{equation}\label{hadwiger}
\int_{G_d} \varphi (K\cap gM) \mu (dg) = \sum_{k=0}^d \varphi_{d-k}(K)V_k(M)
\end{equation}
with certain coefficients $\varphi_{d-k}(K)$ which are functionals in $K$ given by Crofton-type integrals. Here, $G_d$ is the group of (proper) rigid motions and $\mu$ the (suitably normalized) invariant measure on $G_d$. Note that translation invariance of $\varphi$ is not required here. A local version of \eqref{hadwiger} was proved by Schneider \cite{Schn94} (see also \cite[Section 5.3, Note 5]{SW}). In the following, we give an alternative proof of \eqref{hadwiger} and its local variant for local (hence translation invariant) functionals $\varphi$ on ${\cal K}$ with local extension $\Phi\ge 0$  (by Theorem \ref{th1}, $\varphi$ is additive).

\begin{theorem} Let $\varphi$ be a local functional on ${\cal K}$ with local extension $\Phi\ge 0$ and with  $j$-homogeneous parts $\varphi^{(j)}, \Phi^{(j)}$, $j=0,...,d$. Then there are local functionals $\varphi^{(j)}_m$ on ${\cal K}$ with local extension $\Phi^{(j)}_m\ge 0$, $0\le j\le m\le d$, such that
\begin{align}\label{hadwiger2a}
\int_{G_d} &\Phi^{(j)} (K\cap gM,A\cap gB) \mu (dg)\nonumber\\
&= \Phi^{(j)}(K,A)\lambda_M(B)+\sum_{m=j+1}^{d} \Phi_{m}^{(j)}(K,A)\Phi_{d+j-m}(M,B)
\end{align}
and
\begin{equation}\label{hadwiger2}
\int_{G_d} \varphi^{(j)} (K\cap gM) \mu (dg)
= \varphi^{(j)}(K)V_{d}(M)+\sum_{m=j+1}^{d} \varphi_{m}^{(j)}(K)V_{d+j-m}(M)
\end{equation}
for all $K,M\in{\cal K}$, all Borel sets $A,B\in{\cal B}$ and $j=0,\dots ,d$.
\end{theorem}

\begin{proof} It is sufficient to prove the local version \eqref{hadwiger2a}. Then,  \eqref{2-case} yields
\begin{align*}
\int_{G_d} &\Phi^{(j)}(K \cap gM,A\cap gB) \mu (dg)\\
&=\int_{SO_d}  \int_{{\mathbb R}^d} \Phi^{(j)}(K \cap (\vartheta M)^x,A\cap(\vartheta B)^x)\, \lambda(dx)\nu (d\vartheta )\\ 
&= \Phi^{(j)}(K,A) \lambda_M(B) + \sum_{m=j+1}^{d} \int_{SO_d}\Phi_{m,d+j-m}^{(j)}  
(K,\vartheta M;A\times\vartheta B)\nu (d\vartheta )  .
\end{align*}
Here, $SO_d$ is the rotation group with invariant (probability) measure $\nu$.

For polytopes $P,Q$, \cite[Theorem 5.1]{W09} shows that
\begin{align*} &\Phi_{m,d+j-m}^{(j)}(P,\vartheta Q;A\times\vartheta B)\\ &\quad = \sum_{F \in {\cal F}_{m}(P)} \sum_{G \in {\cal F}_{d+j-m}(Q)} 
f_j(n(P,\vartheta Q;F,\vartheta G)) [F, \vartheta G]\lambda_F(A)\lambda_G(B). 
\end{align*}
Hence, we need to evaluate the integral
\begin{align*}
 &\int_{SO_d} f_j(n(P,\vartheta Q;F,\vartheta G)) [F, \vartheta G]\nu (d\vartheta ) \\
&\quad = \int_{SO_d} f_j((N(P,F) +\vartheta N(Q,G))\cap S^{d-1}) [L_1, \vartheta L_2]\nu (d\vartheta ),
\end{align*}
where $L_1, L_2$ are the subspaces orthogonal to $F$ resp. $G$ (see \cite[pp. 191-193]{SW}, for this and the following arguments). We consider, more generally
$$
I(a,b) = \int_{SO_d} f_j((\breve{a}+\vartheta \breve{b})\cap S^{d-1}) [L_1, \vartheta L_2]\nu (d\vartheta )
$$
for closed spherically convex sets $a\subset S^{d-1}\cap L_1, b\subset S^{d-1}\cap L_2$. Here, $\breve{C}$ denotes the cone generated by a set $C\subset S^{d-1}$. It is immediate that $I(a,\eta b) = I(a,b)$ for each rotation $\eta$ leaving $L_2^\bot$ fixed. Moreover, since we assumed $\Phi \ge 0$, we have $f_j\ge 0$. Hence, for fixed $a$, the functional
$$ I_a : b\mapsto I(a,b)$$
is $\ge 0$ and rotational invariant on $\{ b\in {\wp}_{d-j-1}^{d-1} : b \subset L_2\}$. Since $f_j$ is simple, we get similarly to \cite[p. 192]{SW} that $I_a$ is simple (and additive). Theorem 14.4.7 in \cite{SW} now shows that $I_a(b) = c_{L_1}^{(j)}(a) \sigma^{(L_2)}(b)$, where $\sigma^{(L_2)}$ is the spherical Lebesgue measure in $L_2$ and $c_{L_1}^{(j)}(a)\ge 0$ is a constant depending on $a, L_1$ and $f_j$ (but not on $L_2$). Hence, if we define
$$
\Phi^{(j)}_m(P,\cdot) = \sum_{F \in {\cal F}_{m}(P)}  c_{F^\bot}^{(j)}(n(P,F)) \lambda_F ,
$$
it follows that
\begin{align*}
\int_{G_d}& \Phi^{(j)}(P \cap gQ,A\cap gB) \mu (dg)\\
&=\Phi^{(j)}(P,A) \lambda_Q(B) + \sum_{m=j+1}^{d} \Phi^{(j)}_m(P,A) \sum_{G \in {\cal F}_{d+j-m}(Q)} 
\sigma^{(F^\bot)}(n(Q,G)) \lambda_G(B)\cr
&= \Phi^{(j)}(P,A) \lambda_Q(B) + \sum_{m=j+1}^{d} \Phi^{(j)}_m(P,A)  \Phi_{d+j-m}(Q,B)
\end{align*}
due to the representation of the curvature measure $ \Phi_{d+j-m }(Q,\cdot)$ of polytopes  $Q$ (see \cite[eq. (4.22)]{S}).

From the argument in \cite[p. 192]{SW}, we also get that $a\mapsto I(a,b)$ is simple and additive (for fixed $b$), hence 
$c_{L_1}^{(j)}$ is simple and additive. Thus, if we define $\varphi^{(j)}_m(P)=\Phi^{(j)}_m(P,\R^d)$, then  $\varphi^{(j)}_m$ is a translation invariant valuation on $\cal P$ with local extension $\Phi^{(j)}_m$. 

The extension of \eqref{hadwiger2a} and \eqref{hadwiger2} to arbitrary bodies $K,M\in{\cal K}$ follows now as in the proof of Theorem \ref{trans3}. Finally, the continuity of $\Phi_{m}^{(j)}$ and $ \varphi_{m}^{(j)}$ can be obtained from the continuity of $\Phi^{(j)}$ and $ \varphi^{(j)}$ using the homogeneity properties of the former functionals. Thus, $ \varphi_{m}^{(j)}$ is a local functional on $\cal K$ with local extension $ \Phi_{m}^{(j)}$. 
Note that the additivity of $ \varphi_{m}^{(j)}$ (and $ \Phi_{m}^{(j)}$) would follow now also from Theorem \ref{th1}.
\end{proof}

\section{Applications to Boolean models}

In this section, we consider a Boolean model $Z$ in $\R^d$ with convex grains and assume that the underlying (Poisson) particle process $X$ has a translation regular and locally finite intensity measure $\Theta$ (see \cite[Section 11.1]{SW}, for details). Thus, 
\begin{equation}\label{16.1.1}
\Theta (A) = \int_{{\cal K}_0} \int_{{\mathbb R}^d} {\bf 1}_A(K+x) \eta (K,x)\,\lambda (dx)\, {\mathbb Q} (d K),\qquad A\in{\cal B}({\cal K}),
\end{equation}
where ${\cal K}_0$ denotes the set of convex bodies with circumcenter at the origin, ${\mathbb Q}$ is a probability measure on ${\cal K}_0$ and $\eta\ge 0$ is a measurable function on ${\cal K}_0\times \R^d$. If $\eta$ does not depend on $K$, the {\it spatial intensity function} $\eta$ and the {\it grain distribution} $\Q$ are uniquely determined by \eqref{16.1.1}.  By $X^k_{\not=}$ we denote the process of $k$-tuples $(K_1,\dots ,K_k)$ of pairwise different particles $K\in X$, $k=1,2,\dots$.

In the following, we concentrate on a standard functional $\varphi$ with local extension $\Phi$ and assume $\Phi\ge 0$, for simplicity. Since the mixed functionals $\varphi^{(j)}_{m_1,...,m_k}(K_1,\dots, K_k)$ are continuous in each variable $K_i$, they are {\it locally bounded} (that is, bounded on each set $\{K\in{\cal K} : K\subset cB^d\}, c>0$, where $B^d$ is the unit ball).
Therefore, $\varphi$ and the mixed functionals  $\varphi^{(j)}_{m_1,...,m_k}$ meet the requirements in Sections 9.2 and 11.1 of \cite{SW} and we do not need an extra integrability condition here. We only mention that the local finiteness of $\Theta$, which we generally assume, is equivalent to
\begin{align}\label{integrability2}
&\int_{{\cal K}_0}\int_{{\mathbb R}^d}{\bf 1}\{K^{x}\cap C\not=\emptyset\} \eta (K,x)\,\lambda (dx)\,{\mathbb Q}(dK)<\infty ,\quad\quad\end{align}
for any compact $C\subset\R^d$ (see \cite[(11.4)]{SW}).

 The following result is the analog of Corollary 11.1.4 in \cite{SW} and follows in the same way from Theorem \ref{trans3} above (compare also the proof of Theorem 7.1 in \cite{W09}).

 \begin{theorem}\label{bm1} Let $X$ be a Poisson process of convex particles in ${\mathbb R}^d$ with translation regular and locally finite intensity measure, let $k\in{\mathbb N}$, $j\in\{ 0,\dots ,d\}$ and $m_1,\dots,m_k\in \{j,\dots,d\}$ with 
$$
\sum_{i=1}^k m_i = (k-1)d+j.
$$
Let $\varphi$ be a standard functional on $\cal K$  with local extension $\Phi\ge 0$ and let $\Phi^{(j)}_{m_1,...,m_k}$ be the corresponding mixed kernel (which exists by Theorem 2.3). 

Then, 
$${\mathbb E} \sum_{(K_1,\dots ,K_k)\in X^k_{\not=}} \Phi^{(j)}_{m_1,\dots ,m_k}(K_1,\dots ,K_k;\cdot )$$ 
is a locally finite measure on $({\mathbb R}^d)^k$ which is absolutely continuous with respect to $\lambda^k$, and a density is given by  
\begin{eqnarray*}
&&\overline \varphi_{ m_1,\dots,m_k}^{(j)}(X,\dots ,X;z_1,\dots ,z_k) \\ 
&&= \int_{{\cal K}_0}\dots\int_{{\cal K}_0}\int_{({\mathbb R}^d)^k} \eta(K_1,z_1-x_1)\cdots \eta(K_k,z_k-x_k)\,\\
&& \hspace*{4mm}\times\,\Phi^{(j)}_{m_1,\dots,m_k}(K_{1},\dots,K_{k};d (x_1,\dots ,x_k))\,{\mathbb Q} (d K_1)\cdots \,{\mathbb Q} (d K_k)
\end{eqnarray*}
for $\lambda^k$-almost all $(z_1,\dots,z_k)\in({\mathbb R}^d)^k$.
\end{theorem}

Here, for $k=1$, we have $\Phi^{(j)}_j = \Phi^{(j)}$ and correspondingly write $\varphi^{(j)}(X,\cdot)$ for $\varphi^{(j)}_j(X;\cdot)$.

Combining this result with \cite[Theorem 11.1.2]{SW}, we obtain, by a similar proof as in \cite{SW} and analogously to the proof of Theorem 8.1 in \cite{W09}, the following extension of \cite[Theorem 11.1.3]{SW} and of \cite[Theorem 8.1]{W09}. Here, for a standard functional $\varphi$ with local extension $\Phi$, we make use of the fact that $\Phi$ extends to the extended convex ring as a signed Radon measure. Therefore, $\Phi(Z,\cdot )$ is a random signed Radon measure and its expectation ${\mathbb E}\Phi(Z,\cdot )$ is a signed Radon measure.

\begin{theorem}\label{bm2}
Let $Z$ be a Boolean model in ${\mathbb R}^d$ with convex grains and let $\varphi$ be a standard functional with local extension $\Phi\ge 0$. 

Then, for $j=0,...,d$, the signed Radon measure ${\mathbb E}\Phi^{(j)}(Z,\cdot )$ is absolutely continuous with respect to $\lambda$. For $\lambda$-almost all $z$, its density $\overline \varphi^{(j)}(Z,\cdot)$ satisfies
\begin{equation*}
\overline \varphi^{(d)}(Z,z) = c_d\left( 1 -\E^{-\overline V_d(X,z)}\right),
\end{equation*} 
$$
\overline \varphi^{(d-1)}(Z,z) = \E^{-\overline V_d(X,z)} \overline \varphi^{(d-1)}(X,z),
$$
and
\begin{align*}
\overline  \varphi^{(j)}(Z,z)&=\E^{-\overline V_d(X,z)}\left(\overline \varphi^{(j)}(X,z) -\sum_{s=2}^{d-j}\frac{(-1)^{s}}{s!}\right.\\
&\hspace*{4mm}\times\,\sum_{m_1,\dots, m_s=j+1\atop m_1+\dots +m_s=(s-1)d+j}^{d-1} \overline \varphi^{(j)}_{m_1,\dots ,m_s}(X,\dots,X;z,\dots ,z)\Biggr),
\end{align*}
for $j=0,\dots ,d-2.$
\end{theorem}

The densities for $X$ are defined as in Theorem 6.1.

In comparison to Theorem 8.1 of \cite{W09}, we did not assume $c_d=1$ here and therefore had to distinguish $\overline \varphi^{(d)}(X,z)$ and $\overline V_d(X,z)$, which resulted in the factor $c_d$ in the first equation. This covers also the case $c_d=0$.

If $Z$ is stationary, then we do not have to assume that a local extension $\Phi$ exists. The local finiteness condition \eqref{integrability2} then can be relaxed to
\begin{align}\label{integrability4}
\int_{{\cal K}_0} V_d(K+ B^d)\,{\mathbb Q} (dK)<\infty .\quad\quad\end{align}

\begin{koro}\label{CorBM}
Let $Z$ be a stationary Boolean model in ${\mathbb R}^d$ with convex grains and let $\varphi$ be a standard functional. Then, 
$$
\overline \varphi^{(d)}(Z) = c_d\left(1 -\E^{-\overline V_d(X)}\right),
$$
$$
\overline \varphi^{(d-1)}(Z) = \E^{-\overline V_d(X)} \overline \varphi^{(d-1)}(X),
$$
and
\begin{eqnarray*}
\overline  \varphi^{(j)}(Z)&=&\E^{-\overline V_d(X)}\left(\overline \varphi^{(j)}(X) -\sum_{s=2}^{d-j}\frac{(-1)^{s}}{s!}\right.\\
&&\hspace*{4mm}\times\,\sum_{m_1,\dots, m_s=j+1\atop m_1+\dots +m_s=(s-1)d+j}^{d-1} \overline \varphi^{(j)}_{m_1,\dots ,m_s}(X,\dots,X)\Biggr),
\end{eqnarray*}
for $j=0,\dots ,d-2.$
\end{koro}

Here, $\overline  \varphi^{(j)}(Z)$ is the density of the $j$-homogeneous part of $\varphi$, additively extended to the convex ring and defined by
$$
\overline{\varphi}^{(j)}(Z) = \lim_{r \to \infty} \frac
{{\mathbb E}\, \varphi^{(j)}(Z \cap rW)}{V_{d}(rW)}
$$
for any $W\in{\cal K}$ with $V_d(W)>0$ (see \cite[Theorem 9.2.1]{SW}). Note that this definition does not require the existence of a local extension $\Phi^{(j)}$ of $\varphi^{(j)}$. But if $\Phi^{(j)}$ exists, the definition of $\overline{\varphi}^{(j)}(Z)$ as a limit yields the same value as the definition in Theorem 3.2 as a (constant) density with respect to $\lambda$. The corresponding density for $X$ can also be defined as a limit or directly as
$$
\overline{\varphi}^{(j)}(X) = \gamma \int_{{\cal K}_0} \varphi^{(j)}(K)\,{\mathbb Q} (d K)
$$
(here, $\gamma$ is the intensity of the (stationary) process $X$) and the mixed densities are
$$
\overline \varphi_{ m_1,\dots,m_k}^{(j)}(X,\dots ,X)  
= \gamma^k\int_{{\cal K}_0}\dots\int_{{\cal K}_0}\varphi^{(j)}_{m_1,\dots,m_k}(K_{1},\dots,K_{k})\,{\mathbb Q} (d K_1)\cdots \,{\mathbb Q} (d K_k).
$$

Theorem \ref{bm2}  generalizes Theorems 9.1.2 and 11.1.2 in \cite{SW}, since it presents a decomposition of the associated kernels considered in \cite{SW}, in analogy to the intrinsic volumes. Corollary \ref{CorBM} even pushes this further since it allows general standard functionals without making use of a local extension. 

We shortly comment on the situation where the Boolean model $Z$ is stationary and isotropic. Then $\mathbb Q$ is rotation invariant, which yields
\begin{align*}
&\overline \varphi_{ m_1,\dots,m_k}^{(j)}(X,\dots ,X)  \\
&\ = \gamma^k\int_{{\cal K}_0}\dots\int_{{\cal K}_0}\int_{SO_d}\varphi^{(j)}_{m_1,\dots,m_k}(\vartheta K_{1},K_2,\dots,K_{k})\,\nu( d\vartheta)\,{\mathbb Q} (d K_1)\cdots \,{\mathbb Q} (d K_k)
\end{align*}
by Fubini's theorem.
The mapping 
$$
K\mapsto \int_{SO_d}\varphi^{(j)}_{m_1,\dots,m_k}(\vartheta K,K_2,\dots,K_{k})\,\nu( d\vartheta) ,
$$
for $K\in\cal K$, 
is invariant under proper rigid motions, continuous, additive and homogeneous of degree $m_1$. Hence, by Hadwiger's characterization theorem (\cite[Theorem 6.4.14]{S}),
$$
\int_{SO_d}\varphi^{(j)}_{m_1,\dots,m_k}(\vartheta K,K_2,\dots,K_{k})\,\nu( d\vartheta) = c^{(j)}_{m_2,\dots,m_k}(K_2,\dots,K_{k}) V_{m_1}(K)
$$
for $K\in\cal K$ with a constant $c^{(j)}_{m_2,\dots,m_k}(K_2,\dots,K_{k})$ which is translation invariant and continuous in $K_2,\dots ,K_k$ and depends additively and homogeneously (of degree $m_i$) on $K_i$. Using the isotropy of $X$ again, we can thus repeat this construction and split the mixed density further. After $k$ steps, we end up with
\begin{align*}
\overline \varphi_{ m_1,\dots,m_k}^{(j)}(X,\dots ,X) 
&= c(\varphi^{(j)}_{m_1,\dots,m_k})\overline V_{m_1}(X)\cdots \overline V_{m_k}(X)
\end{align*}
with a constant $c(\varphi^{(j)}_{m_1,\dots,m_k})$ which depends only on $\varphi^{(j)}$ and the parameters $m_1,\dots ,m_k$. As a consequence, Corollary \ref{CorBM}, for a stationary and isotropic Boolean model $Z$, reduces to the result for the intrinsic volumes $V_j$ (see \cite[Theorem 9.1.4]{SW} but with constants depending on $\varphi^{(j)}$. This corresponds to Theorem 9.1.3 in \cite{SW}. 

We also remark that the translative formulas for valuations $\varphi$ on ${\cal K}$, as well as for local extensions $\Phi$, generalize to the {\it convex ring} ${\cal R}$ immediately, using the fact that continuous additive functionals on convex bodies have an additive extension to {\it polyconvex sets} (finite unions of convex bodies) (compare the remarks in \cite[Section 10]{W09} and \cite[p. 190]{SW}). This allows to generalize also the formulas from Section 6 to Poisson processes on ${\cal R}$ and to Boolean models with polyconvex grains, provided the integrability conditions \eqref{integrability2} and \eqref{integrability4} are modified appropriately. Since  the grain distribution $\mathbb Q$ is now concentrated on the set ${\cal R}_0$ of polyconvex sets with center of the circumsphere at the origin, \eqref{integrability2} has to be replaced by
\begin{align*}
&\int_{{\cal R}_0}\int_{{\mathbb R}^d} 2^{N(K)}{\bf 1}\{K^{x}\cap C\not=\emptyset\} \eta (K,x)\,\lambda (dx)\,{\mathbb Q} (dK)<\infty ,\quad\quad\end{align*}
for any compact $C\subset\R^d$, where $N(K)$ is the minimal number of convex bodies $K_i$ with $K=\bigcup_{i=1}^{N(K)}K_i$ (see \cite[(9.17)]{SW}). Again, a condition on the mixed functionals, as it was required in \cite[(10.1)]{W09}, is not necessary here (notice that in \cite[(10.1)]{W09} the set ${\cal P}_0$ in the range of integration has to be corrected to $({\cal U}_{GP}({\cal P}))_0$).  In the stationary case, the condition now reads
\begin{align*}
\int_{{\cal R}_0} 2^{N(K)} V_d(K+ B^d)\,{\mathbb Q} (dK)<\infty . \end{align*}

\section{Some special cases}

We now discuss several special cases of valuations $\varphi$ (with local extensions) and show which translative integral formulas and corresponding expectation formulas for Boolean models are induced. Thus, we will recover some formulas from the literature but we shall also obtain new results. We concentrate on convex bodies and Boolean models with convex grains, the extension to polyconvex sets is simple, following the remark made at the end of the previous section.

As a first case, we consider the {\it mixed volume} $\varphi (K) = V(K[j],M_{j+1},\dots,M_d)$, for fixed bodies $M_{j+1},\dots,M_d\in\cK$. It follows from the properties of the intrinsic volume $V_j$ (which corresponds to the case $M_{j+1}= \dots =M_d=B^d$) that $\varphi$ is a standard functional. The iterated translative integral formula from Theorem \ref{trans4} then reads
\begin{align}
&  \int_{({\mathbb R}^d)^{k-1}}
V(K_1\cap K_2^{x_2}\cap \dots \cap K_k^{x_k}[j], M_{j+1},\dots,M_d)\, \lambda^{k-1} (d(x_2,\dots ,x_k))\nonumber \\
& = \sum_{{m_1,\dots,m_k=j}\atop{m_1+\dots +m_k=(k-1)d+j}}^d
V^{(j)}_{m_1,\dots,m_k}(K_1,\dots,K_k;M_{j+1},\dots,M_d),\label{mixedvol}
\end{align}
with mixed functionals $V^{(j)}_{m_1,\dots,m_k}(K_1,\dots,K_k;M_{j+1},\dots,M_d)$, and the corresponding mean value formulas for stationary Boolean models $Z$ are of the form
\begin{align*}
\overline V(Z[j],M_{j+1},\dots,&M_d) =\E^{-\overline V_d(X)}\left(\overline V(X[j],M_{j+1},\dots,M_d) -\sum_{s=2}^{d-j}\frac{(-1)^{s}}{s!}\right.\\
&\times\,\sum_{m_1,\dots, m_s=j+1\atop m_1+\dots +m_s=(s-1)d+j}^{d-1} \overline V^{(j)}_{m_1,\dots,m_s}(X,\dots,X;M_{j+1},\dots,M_d) \Biggr).
\end{align*}
This holds for $j=0,\dots ,d-1$ and we remark that the volume $V_d$ is the special case $j=d$ of the mixed volume. Moreover, for $j=d-1$ the double sum in the above formula disappears. For $M_{j+1}=\dots =M_d=M$, these formulas are given in \cite{W01} (see also \cite[Corollary 11.1.2]{SW}).

In general, the question whether $K\mapsto V(K[j],M_{j+1},\dots,M_d)$ has a local extension seems to be open. In the special case, where $M_d=B^d$ and $M_{j+1},\dots ,M_{d-1}$ are strictly convex, a local extension is given (up to a constant) by 
\begin{align*}
\Phi(K[j],M_{j+1},&\dots,M_{d-1},B^d,\cdot)\\
&=C^{(d-1)}_{j,1,\dots,1}(K,M_{j+1},\dots,M_{d-1};\cdot\times M_{j+1}\times\cdots\times M_{d-1}) ,
\end{align*} 
where $C^{(d-1)}_{j,1,\dots,1}(K,M_{j+1},\dots,M_{d-1};\cdot)$ is the mixed curvature measure 
introduced and studied in \cite{KW99} (see also \cite{Hug} and \cite{HL}, for the related notion of relative curvature measures). This implies a corresponding local integral formula coming from Theorem \ref{trans3} and a mean value formula for Boolean models (without a stationarity assumption) as the outcome of Theorem \ref{bm2}. We do not copy these results here.

As a next case, we consider the {\it (centered) support function} $\varphi (K) = h^\ast (K,\cdot )$. This is a standard functional which is homogeneous of degree 1 with values in the Banach space of centered continuous functions on $S^{d-1}$. To fit this case into our framework, we may apply the results for standard functionals point-wise, that is, for $h^\ast (K,u), u\in S^{d-1}$. The iterated translative formula then reads 
\begin{align}
&  \int_{({\mathbb R}^d)^{k-1}}
h^\ast (K_1\cap K_2^{x_2}\cap \dots \cap K_k^{x_k},\cdot)\, \lambda^{k-1} (d(x_2,\dots ,x_k))\nonumber \\
& = \sum_{{m_1,\dots,m_k=1}\atop{m_1+\dots +m_k=(k-1)d+1}}^d
h^\ast_{m_1,\dots,m_k}(K_1,\dots,K_k,\cdot),\label{suppf}
\end{align}
with mixed support functions $h^\ast_{m_1,\dots,m_k}(K_1,\dots,K_k,\cdot)$. This integral formula was studied in \cite{W95} and \cite{GW03} where it was also shown that the mixed function $h^\ast_{m_1,\dots,m_k}(K_1,\dots,K_k,\cdot)$ is indeed a support function in the case $k=2$ (a proof for general $k$ was given in \cite{S?}). 
The corresponding mean value formula for stationary Boolean models $Z$ is of the form
\begin{align}
\overline {h^\ast}(Z,\cdot) &=\E^{-\overline V_d(X)}\left(\overline {h^\ast}(X,\cdot) -\sum_{s=2}^{d-1}\frac{(-1)^{s}}{s!}\right.\nonumber\\
&\hspace*{4mm}\times\,\sum_{m_1,\dots, m_s=2\atop m_1+\dots +m_s=(s-1)d+1}^{d-1} \overline {h^\ast}_{m_1,\dots,m_s}(X,\dots,X,\cdot ) \Biggr).\label{suppfbm}
\end{align}
The specific (centered) support function $\overline {h^\ast}(Z,\cdot)$ on the left side was introduced and studied in \cite{W94}.

Again, there is a local extension of $K\mapsto h^\ast(K,u)$ given by the mixed measure $\phi_{1,d-1}^{(0)}(K,u^+;\cdot\times \beta(u))$ where $u^+$ is the closed half-space with outer normal $u$ and $\beta (u)$ is a Borel set in the hyperplane $u^\perp$ (bounding $u^+$) with measure $\lambda_{u^\perp}(\beta(u))=1$ (see \cite{W95, GW03}). The corresponding iterated translation formula for this mixed measure is a consequence of Theorem \ref{trans3}, but it also follows from the general results in \cite[Section 6.4]{SW}. The formula which arises from Theorem \ref{bm2} and which is the version of \eqref{suppfbm} in the non-stationary case is not in the literature yet, but since it is quite similar to \eqref{suppfbm}, we skip it here.

Next, we consider the {\it area measure} map $\Psi_j : K\mapsto \Psi_j(K,\cdot)$. It is a translation invariant additive and measure-valued functional which is continuous with respect to the weak topology of measures. It  has a local extension given by the support measure map $\Lambda_j : K\mapsto \Lambda_j(K,\cdot)$, the latter being a measure on $\R^d\times S^{d-1}$ which is concentrated on the (generalized) normal bundle $\Nor K$ of $K$. To fit these measure-valued notions into our results, we cannot consider them point-wise, for a given Borel set, since this would not yield a continuous valuation. However, we can apply our results to the integral 
$$\varphi_f (K) = \int_{S^{d-1}} f(u) \Psi_j(K,du)$$
with a centered continuous  function $f$ on $S^{d-1}$ (and similarly for the support measure). $\varphi_f (K)$ is then a standard functional in $K$, but also a continuous linear functional in $f$, for each fixed $K$. Since these properties carry over to the mixed functionals, we can use the Riesz theorem to obtain formulas for mixed area measures from the results in Sections 4 and 6. The resulting translative formulas were originally obtained in \cite{Hug}  and the mean value formulas for Boolean models are given in \cite{Hoe}. We abstain from copying these results here. 

Instead, we use the functional analytic approach just described in a similar situation, for {\it flag measures} of convex bodies, where corresponding formulas are not available yet. We first describe the underlying notions concerning flag manifolds. Recall that $G(d,j)$ denotes the Grassmannian of $j$-dimensional subspaces (which we supply with the invariant probability measure $\nu_j$) and define corresponding flag manifolds by
$$
F(d,j) = \{(u,L) : L\in G(d,j), u\in L\cap S^{d-1}\}
$$
and
$$
F^\perp (d,j) = \{ (u,L) : L\in G(d,j), u\in L^\perp\cap S^{d-1}\} .
$$
Both flag manifolds carry natural topologies (and invariant Borel probability measures) and $F(d,d-j)$ and $F^\perp(d,j)$ are homeomorphic via the orthogonality map $\rho : (u,L)\mapsto (u, L^\bot)$.  We define a flag measure $\psi_j (K,\cdot)$ as a projection mean of area measures,
\begin{equation}\label{flagprojformula}
\psi_j (K,A) = \int_{G(d,j+1)}\int_{S^{d-1}\cap L} {\bf 1}\{ (u,L^\bot\vee u)\in A\} \Psi'_j(K|L,du) \nu_{j+1}(dL) 
\end{equation}
for a Borel set $A\subset F(d,d-j)$, where $L^\bot\vee u$ is the subspace generated by  $L^\bot$ and the unit vector $u$ and where the prime indicates the area measure calculated in the subspace $L$ (for the necessary measurability properties needed here and in the following, we refer to \cite{Hind}). Using the homeomorphism $\rho$, we can replace $\psi_j (K,\cdot)$ by a measure $\psi_j^\bot (K,\cdot)$ on $F^\perp(d,j)$ given by
\begin{equation}\label{flagprojformula2}
\psi_j^\bot (K,A) = \int_{G(d,j+1)}\int_{S^{d-1}\cap L} {\bf 1}\{ (u,L\cap u^\bot)\in A\} \Psi'_j(K|L,du) \nu_{j+1}(dL) .
\end{equation}
These two (equivalent) versions of the same flag measure are motivated by the fact that their images under the map $(u,L)\mapsto u$ are in both cases the $j$th order area measure  $\Psi_j(K,\cdot)$. Both measures, $\psi_j (K,\cdot)$ and $\psi_j^\bot (K,\cdot)$ have a local version $\lambda_j(K,\cdot)$, respectively $\lambda_j^\bot (K,\cdot)$, which is obtained by replacing in \eqref{flagprojformula} and \eqref{flagprojformula2} the area measure $\Psi'_j(K|L,\cdot)$ by the support measure $\Lambda'_j(K|L,\cdot)$ (see \cite[Theorem 4]{HTW}). In the following, we concentrate on $\psi_j (K,\cdot)$, formulas for the other representation $\psi_j^\bot (K,\cdot)$ follow in a similar way. 

The measure $\psi_j (K,\cdot)$ is centered in the first component,
$$
\int_{F(d,d-j)} u \psi_j(K,d(u,L)) = 0,
$$
as follows from the corresponding property of area measures. Let $C_0(F(d,d-j))$ be the Banach space of continuous functions on $F(d,d-j)$, which are centered in the first component, and choose $f\in C_0(F(d,d-j))$. Then, 
$$\varphi_f : K\mapsto \int_{F(d,d-j)} f(u,L) \psi_j(K,d(u,L))$$
is a local standard functional on $\cK$. Consequently, we obtain the iterated translation formula
\begin{align}
  \int_{({\mathbb R}^d)^{k-1}}
\varphi_f(K_1\cap &K_2^{x_2}\cap \dots \cap K_k^{x_k})\, \lambda^{k-1} (d(x_2,\dots ,x_k))\nonumber \\
& = \sum_{{m_1,\dots,m_k=j}\atop{m_1+\dots +m_k=(k-1)d+j}}^d
\varphi^{(j)}_{f,m_1,\dots,m_k}(K_1,\dots,K_k),\label{flagmeas}
\end{align}
with mixed functionals $\varphi^{(j)}_{f,m_1,\dots,m_k}(K_1,\dots,K_k)$. For fixed bodies $K_1,\dots , K_k$, the left side is a continuous linear functional on $C_0(F(d,d-j))$, if we let $f$ vary. Namely, $f\mapsto \varphi_f(K_1\cap K_2^{x_2}\cap \dots \cap K_k^{x_k})$ is continuous and linear, for each $x_1,\dots ,x_k$, and this carries over to the integral. Replacing $K_1,\dots ,K_k$ by $\alpha_1K_1,\dots ,\alpha_kK_k, \alpha_i>0,$ we use the homogeneity properties of $\varphi^{(j)}_{f,m_1,\dots,m_k}$ to see that the right side is a polynomial in $\alpha_1,\dots ,\alpha_k$. This shows that the coefficients $\varphi^{(j)}_{f,m_1,\dots,m_k}(K_1,\dots,K_k)$ of this polynomial must be continuous linear functionals on $C_0(F(d,d-j))$, too. By the Riesz representation theorem we obtain finite (signed) measures $\varphi^{(j)}_{m_1,\dots,m_k}(K_1,\dots,K_k;\cdot)$ on $F(d,d-j)$ such that
$$
\varphi^{(j)}_{f,m_1,\dots,m_k}(K_1,\dots,K_k) = \int_{F(d,d-j))} f(u,L) \psi^{(j)}_{m_1,\dots,m_k}(K_1,\dots,K_k;d(u,L))
$$
for all $f\in C_0(F(d,d-j))$. The measures are uniquely determined, if we require that they are centered. We call them the {\it mixed flag measures}. Hence we obtain the iterated translation formula for flag measures,
\begin{align}
  \int_{({\mathbb R}^d)^{k-1}}
\psi_j(K_1\cap &K_2^{x_2}\cap \dots \cap K_k^{x_k},\cdot)\, \lambda^{k-1} (d(x_2,\dots ,x_k))\nonumber \\
& = \sum_{{m_1,\dots,m_k=j}\atop{m_1+\dots +m_k=(k-1)d+j}}^d
\psi^{(j)}_{m_1,\dots,m_k}(K_1,\dots,K_k;\cdot).\label{flagmeas2}
\end{align}

Theorem \ref{bm2} then gives us formulas for the specific flag measures $\overline  \psi_j(Z,z;\cdot)$, $j=0,\dots ,d-1$, as measure-valued functions of $z\in\R^d$,
\begin{align}
\overline  \psi_j(Z,z;\cdot)&=\E^{-\overline V_d(X,z)}\left(\overline \psi_j(X,z;\cdot) -\sum_{s=2}^{d-j}\frac{(-1)^{s}}{s!}\right.\nonumber\\
&\hspace*{4mm}\times\,\sum_{m_1,\dots, m_s=j+1\atop m_1+\dots +m_s=(s-1)d+j}^{d-1} \overline \psi^{(j)}_{m_1,\dots ,m_s}(X,\dots,X;z,\dots ,z;\cdot)\Biggr).\label{flagbm}
\end{align}
Notice that both sides are (centered) measures on $F(d,d-j)$ and that the formulas for specific area measures result if we apply the mapping $(u,L)\mapsto u$. For $j=d-1$ the double sum disappears. If $Z$ is stationary, the quantities in \eqref{flagbm} are independent of $z$.

It would be possible to apply our general results also to {\it tensor valuations}, coordinate-wise. However, the condition of translation invariance makes the results less interesting since for tensor valuations a notion of {\it translation covariance} is more natural (see \cite[p. 363]{S}, \cite[p. 198-9]{SW} and \cite{HS15}.) For the special class of Minkowski tensors the above-mentioned formulas for support measures can be used instead. The resulting translation formulas for Minkowski tensors and the tensorial formulas for Boolean models are collected in \cite{Hetal} and in the forthcoming survey \cite{HW15}.

\section{Appendix}
We formulate and prove here a lemma on the simultaneous approximation of convex bodies by polytopes which shows that a continuous functional $\varphi : \cK\to\R$ is additive provided it is additive on $\cP$. The following proof, which replaces a more complicated argument by D. Hug and W. Weil, is due to R. Schneider.

\begin{lemma}\label{Rolf} Let $K_1, \dots ,K_m \in \cK$ be convex bodies such that $K = K_1 \cup\dots\cup K_m$ is convex. Let $\varepsilon > 0$. Then there are polytopes $P_1,\dots, P_m \in \cP$ with $K_i \subset P_i \subset K_i + \varepsilon B^d$ for $i = 1, \dots ,m$
such that $P = P_1\cup\dots\cup P_m$ is convex.
\end{lemma}

\begin{proof} We choose a polytope $Q$ with $Q \subset K \subset Q+\varepsilon B^d$ and to each $i \in \{1, \dots,m\}$  a polytope $R_i$ with $K_i \subset R_i \subset K_i + \varepsilon B^d$. Then we define $Q_i = Q \cap R_i$ for $i = 1, \dots , m$. We have
$$
Q = Q \cap K = \bigcup_{i=1}^m (Q \cap K_i) \subset \bigcup_{i=1}^m (Q \cap R_i) = \bigcup_{i=1}^m Q_i \subset Q,
$$
hence
$$
Q = \bigcup_{i=1}^m Q_i .
$$

We claim that $\delta (K_i,Q_i) \le  \varepsilon.$ 
For the proof, let $x \in K_i$, and let $y = p(Q_i, x)$ (where $p$ denotes the nearest-point map).
If $x \in Q$, then $x \in Q_i$, hence $x = y$. If $x \notin Q$, then $y = p(Q \cap R_i, x) = p(Q, x)$. Since
$x \in K \subset Q + \varepsilon B^d$, we have $\| x - y\| \le  \varepsilon$. Conversely, let $x \in Q_i$, and let $y = p(K_i, x)$. From $x \in R_i \subset K_i + \varepsilon B^d$ it follows that $\| x - y\|  \le  \varepsilon$. This proves the claim.

Since $\delta (K_i,Q_i) \le  \varepsilon$, we have $K_i \subset Q_i + \varepsilon B^d$. We choose a polytope $C$ with $\varepsilon B^d \subset C \subset
2\varepsilon B^d$ and define $P_i = Q_i + C$. Then $K_i \subset P_i \subset K_i + 3\varepsilon B^d$. Moreover, by \cite[(3.1)]{S},
$$
P = P_1 \cup \dots \cup P_m = (Q_1 + C) \cup \dots \cup (Q_m + C) = (Q_1 \cup \dots \cup Q_m) + C = Q + C,
$$
thus $P$ is convex.
\end{proof}

As we mentioned, the lemma implies that a continuous map $\varphi : \cK\to\R$ which is additive on $\cP$ is additive on $\cK$, hence a valuation. In fact, we obtain a stronger result since we need only require that $\varphi$ is weakly additive on $\cP$ and we get that $\varphi$ is fully additive on $\cK$. The latter means that
\begin{equation}\label{fulladd}
\varphi(K)=\sum_{\emptyset\not= I\subset \{1,\dots ,m\}} (-1)^{|I|-1}\varphi (K_I),
\end{equation}
for $K\in\cK, K=\bigcup_{i=1}^m K_i, K_i\in\cK$, where $K_I= \bigcap_{i\in I} K_i$ and $|I|$ is the cardinality of $I$. To obtain this, we first remark that a weakly additive functional on $\cP$ is fully additive on $\cP$ by \cite[Theorem 6.2.3]{S}. We now apply the lemma with $K_i$ replaced by $K_i+(1/2^r)B^d, r\in{\mathbb N},$ and $\varepsilon = 1/2^{r}$ (note that $\bigcup_{i=1}^m (K_i+(1/2^r)B^d) = K+(1/2^r)B^d$). Thus, we obtain polytopes $P^{(r)}_1,\dots ,P^{(r)}_m$ with convex union $P^{(r)}=\bigcup_{i=1}^m P^{(r)}_i$ and such that $K_i+(1/2^r)B^d\subset P_i^{(r)}\subset K_i+(1/2^{r-1})B^d$. Since $P_i^{(r)}$ is a decreasing sequence of polytopes, as $r\to\infty,$ the polytopes $P_I^{(r)},\emptyset\not= I\subset \{1,\dots ,m\},$ also constitute a decreasing sequence with $\bigcap_{r=1}^\infty P_I^{(r)}= K_I$. Thus, Lemma 1.8.2 in \cite{S} implies $P_I^{(r)}\to K_I$ (and similarly $P^{(r)}\to K$) and \eqref{fulladd} follows from the full additivity of $\varphi$ on ${\cal P}$.

Since the full additivity of a functional $\varphi$ on $\cK$ is equivalent to its additive extendability to polyconvex sets (see \cite[Theorem 6.2.1]{S}), we also obtain now that a local  functional $\varphi$ on $\cK$ has an additive extension to the convex ring $U(\cK)$ without using the (more complicated) extension theorem of Groemer (\cite[Theorem 6.2.5]{S}).

\section*{Acknowledgements}

The research of the author has been supported by the DFG project WE 1613/2-2.

\medskip\noindent
The author is grateful to a referee for useful remarks which helped to improve the presentation of the paper.


\begin{thebibliography}{99}

\bibitem{GW03} P. Goodey, W. Weil, Translative and kinematic integral
formulae for support functions II, Geom. Dedicata 99 (2003) 103--125.

\bibitem{Hind} W. Hinderer,  Integral Representations of Projection Functions, PhD Thesis, University of Karlsruhe, Karlsruhe, 2002.

\bibitem{HHW} W. Hinderer, D. Hug, W. Weil, Extensions of translation invariant valuations on polytopes, Mathematika 61 (2015) 236--258.

\bibitem{Hoe} J. H\"orrmann, The Method of Densities for Non-isotropic Boolean Models, PhD Thesis, Karlsruhe Institute of Technology, KIT Scientific Publ., Karlsruhe, 2015.

\bibitem{Hetal} J. H\"orrmann, D. Hug, M. Klatt, K. Mecke, Minkowski tensor density formulas for Boolean models, Adv. in Appl. Math.  55 (2014) 48--85.

\bibitem{HW15} J. H\"orrmann, W. Weil, Valuations and Boolean models, in: E.B.V. Jensen, M. Kiderlen (Eds.) Tensor Valuations and their Applications in Stochastic Geometry and Imaging, 2016, in preparation.

\bibitem{Hug} D. Hug,  Measures, Curvatures and Currents in Convex Geometry, 
Habilitation Thesis, University of Freiburg, Freiburg, 1999. 

\bibitem{HL} D. Hug, G. Last, On support measures in Minkowski spaces and contact distributions in stochastic geometry, Ann. Probab.  28 (2000) 796--850.

\bibitem{HS15} D. Hug, R. Schneider, Tensor valuations and their local versions, in: E.B.V. Jensen, M. Kiderlen (Eds.) Tensor Valuations and their Applications in Stochastic Geometry and Imaging, 2016, in preparation.

\bibitem{HTW} D. Hug, I. T\"urk , W. Weil, Flag measures for convex bodies, in: M. Ludwig et al. (Eds.) Asymptotic Geometric Analysis, Fields Institute Communications, Vol. 68, Springer, 2013, pp. 145--187.  

\bibitem{KW99} M. Kiderlen, W. Weil, Measure-valued valuations and mixed curvature measures of convex bodies, Geometriae Dedicata 76 (1999) 291--329.


\bibitem{McM93} P. McMullen,  Valuations and dissections, in: P.M. Gruber, J.M. Wills (Eds), Handbook  of Convex Geometry, vol. B, North-Holland, Amsterdam, 1993, pp. 933--988.


\bibitem{Schn94} R. Schneider, An extension of the principal kinematic formula of integral geometry, Rend. Circ. Mat. Palermo II, Suppl., 35 (1994) 275--290.

\bibitem{S?} R. Schneider, Mixed polytopes, Discrete Comput. Geom.  29  (2003) 575--593.

\bibitem{S} R. Schneider, Convex Bodies: The Brunn-Minkowski Theory, second expanded edition, Cambridge University Press, Cambridge, 2014.

\bibitem{SW}  R. Schneider, W. Weil, Stochastic and Integral Geometry, Springer, Berlin, 2008.


\bibitem{W94} W. Weil, Support functions on the convex ring in the plane and
support densities of random sets and point processes, 
Rend. Circ. Mat. Palermo II, Suppl, 35 (1994) 323--344.

\bibitem{W95} W. Weil, Translative and kinematic integral formulae for support
functions, Geom. Dedicata 57 (1995) 91--103.

\bibitem{W01} W. Weil, Densities of mixed volumes for Boolean models, Adv. Appl. Probab. 33 (2001) 39--60. 

\bibitem{W09} W. Weil, Integral geometry of translation invariant functionals, I: The polytopal case, Adv. in Appl. Math. 66 (2015) 46--79.



\end{thebibliography}
\end{document}